\documentclass[11pt, reqno]{article}
\usepackage{amsfonts}
\usepackage{amsmath}
\usepackage{amsthm}
\usepackage{amssymb}
\usepackage{lastpage}
\usepackage{fancyhdr}
\usepackage{fancybox}
\usepackage{mathrsfs}

\def\b{\mathbb }

\def\phi{\varphi }
\def\epsilon{\varepsilon}

\swapnumbers

\theoremstyle{plain}
\newtheorem{theorem}{Theorem}[section]
\newtheorem{corollary}[theorem]{Corollary}
\newtheorem{lemma}[theorem]{Lemma}

\theoremstyle{definition}

\newtheorem{remark}[theorem]{Remark}

\numberwithin{equation}{section}

\begin{document}

\title{ Uniform oscillatory behavior of spherical functions of $GL_n/U_n$ at
  the identity and a central limit theorem}
\author{
Michael Voit\\
Fakult\"at Mathematik, Technische Universit\"at Dortmund\\
          Vogelpothsweg 87,
          D-44221 Dortmund, Germany\\
e-mail:  michael.voit@math.tu-dortmund.de}
\date{\today}

\maketitle

\begin{abstract} Let $\mathbb F=\mathbb R$ or $\mathbb C$ and $n\in\b N$.
 Let $(S_k)_{k\ge0}$ be a time-homogeneous random walk on $GL_n(\b F)$  associated with 
an $U_n(\b F)$-biinvariant measure $\nu\in M^1(GL_n(\b F))$. We derive a central limit theorem  for the ordered
  singular spectrum $\sigma_{sing}(S_k)$  with  a  normal distribution
 as limit with  explicit analytic formulas for the drift  vector and  the covariance matrix.
 The main ingredient for the proof will be a
oscillatory result for the spherical functions $\phi_{i\rho+\lambda}$ of
$(GL_n(\b F),U_n(\b F))$. More precisely, we present a necessarily unique
 mapping $m_{\bf 1}:G\to\b R^n$ such that for some  
 constant
$C$ and  all $g\in G$, $\lambda\in\b R^n$,
$$|\phi_{i\rho+\lambda}(g)- e^{i\lambda\cdot m_{\bf 1}(g)}|\le
C\|\lambda\|^2.$$

\end{abstract}



\section{Introduction}

Let $\mathbb F=\mathbb R$ or $\mathbb C$, $n\ge2$ an integer,
 and $G:=GL(n,\mathbb F)$ the general linear group with maximal compact subgroup $K:=U_n(\mathbb F)$. 
 Consider  i.i.d.~$G$-valued random variables $(X_k)_{k\ge1}$ with the common 
$K$-biinvariant distribution $\nu\in M^1(G)$ and the associated 
 $G$-valued random walk $(S_k:= X_1 \cdot X_2\cdots X_k)_{k\ge0}$ with the convention  that $S_0$ is 
 the identity  $I_n$.  Moreover, let 
$$\sigma_{sing}(g)\in \{x=(x_1,\ldots,x_n)\in\mathbb R^n:\> x_1\ge x_2\ge\cdots\ge x_n>0\}$$
denote the singular (or Lyapunov)  spectrum of $g\in G$ where
 the singular values of $g$, i.e.,  square roots of the 
eigenvalues of $gg^*$, are ordered by size. Consider the mapping $\ln\sigma_{sing}$ from $G$
onto the Weyl chamber 
$$W_n:=\{x=(x_1,\ldots,x_n)\in\mathbb R^n:\> x_1\ge x_2\ge\cdots\ge x_n\},$$
with the logarithm 
$\ln(x_1,\ldots,x_n):=(\ln x_1,\ldots,\ln x_n)$.
 We show  that under a natural moment condition,
  the $\mathbb R^n$-valued random variables
\begin{equation}\label{normal}
\frac{1}{\sqrt k}(2\cdot \ln\sigma_{sing}(S_k)-k\cdot m_{\bf 1}(\nu))
\end{equation}
tend for $k\to\infty$  to some $n$-dimensional normal
distribution
 $N(0,\Sigma^2(\nu))$
 where the drift vector $ m_{\bf 1}(\nu)$ and the covariance matrix $\Sigma^2(\nu)$ 
are given explicitely depending on $\nu$. 

This central limit theorem (CLT) can be also seen as follows:
 By polar decomposition of  $g\in G$, the symmetric space $G/K$ can be identified with the cone 
$P_n(\mathbb F)$ of positive definite symmetric or hermitian $n\times n$ matrices via
$$gK\mapsto I(g):= gg^*\in P_n(\mathbb F) \quad\quad(g\in G),$$
where $G$ acts on $P_n(\mathbb F)$ via $a\mapsto gag^*$.  In this way,
 the double coset space $G//K$ can be identified with the Weyl chamber
 $W_n$ via 
$$KgK\mapsto \ln\sigma_{sing}(g)=\frac{1}{2}\ln \sigma( gg^*)$$
where $\sigma$ denotes the spectrum, i.e., the ordered eigenvalues, of a positive definite matrix. 
Therefore, the CLT above may be regarded as a CLT for the spectrum of $K$-invariant random walks on  
$P_n(\mathbb F)$. Such   CLTs  have a long history. CLTs where  $\nu$ is renormalized first into some 
 measure $\nu_k\in M^1(G)$, and then the convergence of the convolution powers $\nu_k^k$ 
 is studied, can be found  e.g. in
\cite{KTS}, \cite{Tu}, \cite{FH}, \cite{Te1}, \cite{Te2},
\cite{Ri}, \cite{G1}, and \cite{G2}. In this case,  so-called dispersions of $\nu$ appear
 as  parameters of the limits, where these dispersions are defined in terms
 of derivatives of the spherical functions of $(G,K)$. 
These dispersions will also appear in our CLT
 in order to describe  $m_{\bf 1}(\nu)$ and  $\Sigma^2(\nu)$. Our CLT is in principle well-known;
 see  Theorem 1 of \cite{Vi}, as well as 
the  CLTs of Le Page \cite{L} and  the monograph \cite{BL}. However, our approach, which directly leads to 
analytic formulas for  drift and covariance, seems to be new for $n>2$.
For  $n=2$, our CLT can be splitted into two one-dimensional parts, namely
 a classical part for the sum
 $\ln\det S_k=\sum_{l=1}^k\ln\det X_l $
 of i.i.d.~random variables, and a CLT   for  $(SL_2(\mathbb F),SU(2,\mathbb F))$.
 The associated spherical functions are the Jacobi functions $\phi_\lambda^{(0,0)}(t)$ and $\phi_\lambda^{(1/2,1/2)}(t)$
 depending on $\b F$ (see \cite{K} for details), and the CLT above for
 $(SL_2(\mathbb F),SU(2,\mathbb F))$ appears as a special case of a CLT of Zeuner \cite{Z}
  for  certain
 Sturm-Liouville hypergroups on $[0,\infty[$. 
The proof of Zeuner depends on some uniform estimate for  the oscillatory behavior of the associated
 multiplicative functions, i.e., the
Jacobi functions here.
This idea was later on transfered to certain random walks on the nonnegative integers
 associated with orthogonal polynomials in \cite{V1}. Moreover, the result of Zeuner 
 \cite{Z} was recently 
slightly improved for  Jacobi functions
in \cite{V2} by using the well-known Harish-Chandra integral representation of the 
 Jacobi functions from \cite{K}.  We here  adopt
 this approach and use the Harish-Chandra integral representation
of the spherical function of $(G,K)$ 
 to derive a uniform estimate for their oscillatory behavior.
 The CLT above then follows  easily.

Let us describe this uniform oscillatory result.
Recapitulate that a $K$-biinvariant continuous function $\phi\in C(G)$ on $G$ is called spherical iff 
$$\phi(g_1)\phi(g_2)=\int_K \phi(g_1kg_2)\> dk$$
 for all $g_1,g_2\in G$ where $dk$ is the normalized Haar measure on $K$.
 It is well-known (see \cite{H1} or \cite{Te2}) that all spherical functions of $(G,K)$ are given by the 
 Harish-Chandra integral
\begin{equation}\label{int-rep}
\phi_{i\rho+\lambda}(g)=\int_K \Delta_1^{i\lambda_1-i\lambda_2}(k^*gg^*k)\cdots
 \Delta_{n-1}^{i\lambda_{n-1}-i\lambda_n}(k^*gg^*k)\Delta_{n}^{i\lambda_n}(k^*gg^*k)\> dk
\end{equation}
where the $\Delta_r$ are the principal minors of order $r$, $\lambda\in\mathbb C^n$, and
 where $\rho=(\rho_1,\ldots,\rho_n)$ is the half sum of roots with
$\rho_l=\frac{d}{2}(n+1-2l)$ with the dimension $d=1,2$ of $\mathbb F$ over $\mathbb R$.
Notice that by (\ref{int-rep}), $\phi_{i\rho}\equiv 1$, and that for $\lambda\in\mathbb R^n$ and $g\in G$,
$|\phi_{i\rho+\lambda}(g)|\le1$.

We now follow the usual approach to the dispersion for $(G,K)$ (see \cite{FH},\cite{Te1}, \cite{Te2},
\cite{Ri}, \cite{G1}, \cite{G2}) and to so-called moment functions on hypergroups in \cite{Z}, \cite{V1}, and
 Section 7.2.2 of \cite{BH}:  For  multiindices $l=(l_1,\ldots,l_n)\in\mathbb N_0^n$ we define
the so
called  moment functions
\begin{align}\label{moment-function}
m_l(g):=&
\frac{\partial^{|l|}}{\partial\lambda^l}\phi_{i\rho-i\lambda}(g)
\Bigl|_{\lambda=0}:=\frac{\partial^{|l|}}{(\partial\lambda_1)^{l_1}\cdots(\partial\lambda_n)^{l_n}}
\phi_{i\rho-i\lambda}(g)
\Bigl|_{\lambda=0}
\notag\\
=&
\int_K (\ln\Delta_1(k^*gg^*k))^{l_1}\cdot\left(\ln\left(\frac{\Delta_2(k^*gg^*k)}{\Delta_1(k^*gg^*k)}\right)\right)^{l_2}
\cdots
\left(\ln\left(\frac{\Delta_n(k^*gg^*k)}{\Delta_{n-1}(k^*gg^*k)}\right)\right)^{l_n}\> dk
 \end{align}
of order $|l|:=l_1+\cdots+l_n$ for $g\in G$. Clearly, the 
last equality follows immediately from  (\ref{int-rep})
 by interchanging integration and  derivatives.
Using the $n$ moment functions $m_l$ of first order $|l|=1$, we form the vector-valued moment function
\begin{equation}\label{m1-vector}
m_{\bf 1}(g):=(m_{(1,0,\ldots,0)}(g),\ldots,m_{(0,\ldots,0,1)}(g))
\end{equation}
of first order.
Moreover, we use the usual scalar product $x\cdot y:=\sum_{l=1}^n x_ly_l$ on $\b R^n$.
We can now formulate the following oscillatory result;
 it will be proved in Section 2.

\begin{theorem}\label{osc}
 There exists a constant
$C=C(n)$
such that for all $g\in G$  and 
and $\lambda\in\b R^n$,
$$|\phi_{i\rho+\lambda}(g)- e^{i\lambda\cdot m_{\bf 1}(g)}|\le
C\|\lambda\|^2.$$
\end{theorem}

The function  $m_{\bf 1}$ is obviously determined uniquely by the property of the theorem.

We  return to the CLT. Similar to collecting the moment functions of first
order in the vector $m_{\bf 1}$, we group the  moment functions of second
order by
\begin{align}\label{m2-matrix}
m_{\bf 2}(g):=&\left(\begin{array}{ccc} m_{1,1}(g)&\cdots& m_{1,n}(g)\\
\vdots &&\vdots \\ m_{n,1}(g)&\cdots& m_{n,n}(g) \end{array}\right)
\\
:=&
\left(\begin{array}{cccc} m_{(2,0,\ldots,0)}(g)&m_{(1,1,0,\ldots,0)}(g)&\cdots&m_{(1,0,\ldots,0,1)}(g)
\\ m_{(1,1,0,\ldots,0)}(g)&m_{(0,2,0,\ldots,0)}(g)&\cdots&m_{(0,1,0,\ldots,0,1)}(g)
\\ \vdots &\vdots&&\vdots
\\ m_{(1,0,\ldots,0,1)}(g)&m_{(0,1,0,\ldots,0,1)}(g)&\cdots&m_{(0,\ldots,0,2)}(g)
\end{array}\right)
\notag\end{align}
for $g\in G$. We show in Section 3 as an easy consequence of
(\ref{moment-function})
that the $n\times n$ matrices $m_{\bf 2}(g)-m_{\bf 1}(g)^t\cdot m_{\bf 1}(g)$ are positive semidefinite.

Now consider  $\nu\in M^1(G)$ such that  the moment functions $m_{j,j}\ge0$ ($j=1,\ldots,n$) are $\nu$-integrable.
 We then say that  $\nu$ admits finite second moments. In this
case,  (\ref{moment-function}) and the Cauchy-Schwarz
inequality yield that all moments of order one and two are 
$\nu$-integrable, and we form the modified expectation vector and covariance matrix 
$$m_{\bf 1}(\nu):=  \int_G m_{\bf 1}(g)\> d\nu\in\b R^n ,\quad\quad
\Sigma^2(\nu):=  \int_G m_{\bf 2}(g)\> d\nu \> -\> m_{\bf 1}(\nu)^t\cdot  m_{\bf 1}(\nu)$$
of $\nu$.
The precise statement of our CLT is now as follows:

\begin{theorem}\label{main-clt}
If   $\nu\in M^1(G)$ is
$K$-biinvariant and admits finite second moments, then
\begin{equation}
\frac{1}{\sqrt k}(2\cdot\ln\sigma_{sing}(S_k)-k\cdot m_{\bf 1}(\nu)) \longrightarrow
N(0,\Sigma^2(\nu))
\end{equation}
 for $k\to\infty$ in distribution.
\end{theorem}

This paper is organized as follows: Section 2 is devoted exclusively to the
proof of Theorem \ref{osc}. 
In Section 3 we then shall present the proof of Theorem \ref{main-clt}. There
we also give  a precise condition on $\nu$ when $\Sigma^2(\nu)$  
 is positive definite.

 We finally remark that
the results of our paper can
be transfered to the Grassmann manifolds $(SO_0(p,q)/(SO(p)\times SO(q))$ and 
$(SU(p,q)/S(U(p)\times U(q))$. In this case, the spherical functions are certain
Heckman-Opdam hypergeometric functions of type BC, for which a Harish-Chandra
integral representation analog to (\ref{int-rep}) is available; see \cite{Sa}
and \cite{RV}. We plan to carry out this in near future.

\section{Proof of the oscillatory behavior of spherical functions}

This section is  devoted to the proof of Theorem \ref{osc} which
 depends on several facts which may be more or less well-known.
As we could not find suitable published references, we include proofs for sake of completeness.
We start with a result about the principal minors $\Delta_r$:

\begin{lemma}\label{pos-coeff}
Let $1\le r\le n$ be integers, $\mathbb F=\mathbb R$ or $\mathbb C$,  and
$u\in U_n(\mathbb F)$. Then
$$\Delta_r(u^*\cdot diag(a_1,\ldots,a_n)\cdot u)=\sum_{1\le i_1<i_2<\ldots<i_r\le n}
 c_{i_1,\ldots,i_r} a_{i_1}\cdot a_{i_2}\cdots \cdot a_{i_r}$$
for all $a_{i_1},a_{i_2},\ldots, a_{i_r}\in \mathbb R$ with coefficients
$c_{i_1,\ldots,i_r}=c_{i_1,\ldots,i_r}(u)$ satisfying $c_{i_1,\ldots,i_r}\ge0$
for $1\le i_1<i_2<\ldots<i_r\le n$ and
$\sum_{1\le i_1<i_2<\ldots<i_r\le n} c_{i_1,\ldots,i_r}=1$.
\end{lemma}

\begin{proof}
Clearly, $h_r(a_1,\ldots, a_n):=\Delta_r(u^*\cdot diag(a_1,\ldots,a_n)\cdot u)$
is a homogeneous polynomial of degree $r$, i.e.,
$$h_r(a_1,\ldots, a_n)=\sum_{1\le i_1\le i_2\le\ldots\le i_r\le n}
 c_{i_1,\ldots,i_r} a_{i_1}\cdot a_{i_2}\cdots \cdot a_{i_r}.$$
We first check that $c_{i_1,\ldots,i_r}\ne0$ is possible only for
$1\le i_1<i_2<\ldots<i_r\le n$. For this consider $i_1,\ldots,i_r$ with
$|\{i_1,\ldots,i_r\}|=:q<r$. By changing the numbering of the variables
$a_1,\ldots,a_n$ (and  of rows and columns of $u$ in an appropriate
way), we  may assume that $\{i_1,\ldots,i_r\}=\{1,\ldots, q\}$. In this case,
$u^*\cdot diag(a_1,\ldots,a_q,0,\ldots,0)\cdot u$ has rank at most $q<r$.
Thus
$$0=h_r(a_1,\ldots,a_q,0,\ldots,0)=\sum_{1\le i_1\le i_2\le\ldots\le i_r\le q}
 c_{i_1,\ldots,i_r} a_{i_1}\cdot a_{i_2}\cdots \cdot a_{i_r}$$
for all $a_1,\ldots,a_q$. This yields $c_{i_1,\ldots,i_r}=0$ for $1\le i_1\le
i_2\le\ldots\le i_r\le q$
and proves that
$$h_r(a_1,\ldots, a_n)=\sum_{1\le i_1<i_2<\ldots<i_r\le n}
 c_{i_1,\ldots,i_r}a_{i_1}\cdot a_{i_2}\cdots \cdot a_{i_r}.$$

For the nonnegativity we again may restrict our attention to
$c_{1,\ldots,r}$. In this case,
$$0\le \left(\begin{array}{cc} I_r&0\\0&0\end{array}\right)\le I_n
\quad\quad\text{and thus}\quad\quad
0\le u^*\left(\begin{array}{cc} I_r&0\\0&0\end{array}\right)u\le I_n$$
w.r.t. the usual ordering of positive semidefinite matrices.
As this inequality holds also for the upper left $r\times r$ block, we obtain
$$c_{1,\ldots,r}=h_r(1,\ldots,1,0,\ldots,0)=
\Delta_r\left(u^*\left(\begin{array}{cc}
  I_r&0\\0&0\end{array}\right)u\right)\ge0.$$
Finally, as
$$\sum_{1\le i_1<i_2<\ldots<i_r\le n} c_{i_1,\ldots,i_r}= h_r(1,\ldots,1)=1,$$
the proof is complete.
\end{proof}

Let us keep the notation of Lemma \ref{pos-coeff}. We now compare
 $h_r(a_1,\ldots, a_n)$ with the homogeneous polynomial
\begin{equation}\label{C-r}
C_r(a_1,\ldots,a_n):=\frac{1}{{n\choose r}}
\sum_{1\le i_1<i_2<\ldots<i_r\le n}a_{i_1} a_{i_2}\cdots \cdot a_{i_r}>0 
\quad\quad(r=1,\ldots,n).
\end{equation}

\begin{lemma}\label{abs}
For all $a_1,\ldots,a_n>0$,
$$0<\frac{C_r(a_1,\ldots,a_n)}{h_r(a_1,\ldots,a_n)}\le \frac{1}{{n\choose r}} 
\sum_{1\le i_1<i_2<\ldots<i_r\le n}c_{i_1,\ldots,i_r}(u)^{-1},$$
where, depending on $u$, on both sides the value $\infty$ is possible.
\end{lemma}

\begin{proof} Positivity is clear by Lemma \ref{pos-coeff}. Moreover,
\begin{align}
C_r(a_1,\ldots,a_n)=&\frac{1}{{n\choose r}}
\sum_{1\le i_1<i_2<\ldots<i_r\le n}a_{i_1} a_{i_2}\cdots \cdot a_{i_r}
\notag\\
\le&\frac{\max_{1\le i_1<i_2<\ldots<i_r\le n}c_{i_1,\ldots,i_r}^{-1}}{{n\choose r}}
\sum_{1\le i_1<i_2<\ldots<i_r\le n}c_{i_1,\ldots,i_r}a_{i_1} a_{i_2}\cdots \cdot a_{i_r}
\notag
\end{align}
which immediately leads to the claim.
\end{proof}

We also need the following observation from linear algebra:

\begin{lemma}\label{glech-det}
Let $u\in U_n(\b C)$ have the block structure
 $u=\left(\begin{array}{cc} u_1&*\\ *&u_2\end{array}\right)$
with quadratic blocks $u_1\in M_r(\mathbb C)$ and $u_2\in M_{n-r}(\mathbb C)$ 
with $1\le r\le n$. Then $|\det u_1|=|\det u_2|$.
\end{lemma}

\begin{proof} W.l.o.g. we may assume $2r\le n$. By the $KAK$-decomposition of  $U_n(\mathbb C)$
with $K=U_{r}(\mathbb C)\times U_{n-r}(\mathbb C)$ (see e.g.~Theorem VII.8.6 of \cite{H2}), we  write
$u$ as
$$u=\left(\begin{array}{cc} a_1&0\\0&b_1\end{array}\right)
\cdot\left(\begin{array}{ccc} c&s&0\\-s&c&0\\0&0&I_{q-2r}\end{array}\right)
\cdot\left(\begin{array}{cc} a_2&0\\0&b_2\end{array}\right)
$$
with $a_1,a_2\in U_r(\mathbb C)$, $b_1,b_2 \in U_{n-r}(\mathbb C)$ and with
 $c=diag(\cos\phi_1,\ldots,\cos\phi_r)$ and  $s=diag(\sin\phi_1,\ldots,\sin\phi_r)$ for suitable
$ \phi_1,\ldots,\phi_r\in\mathbb R$. Therefore, 
$$u_1=a_1ca_2
\quad\quad\text{and}\quad\quad
 u_2=b_1\left(\begin{array}{cc} c&0\\0&I_{q-2r}\end{array}\right)b_2$$
 which immediately implies the claim.
\end{proof}

We shall also need the  following elementary observation:

\begin{lemma}\label{ln-est}
Let  $\epsilon\in]0,1]$, $M\ge1$ and $m\in\b N$. Then there exists a constant $C=C(\epsilon,M,m)>0$
such that for all $z\in]0,M]$,
$$|\ln(z)|^m\le C\left(1+z^{-\epsilon}\right).$$
\end{lemma}

\begin{proof} Elementary calculus yields $|x^\epsilon\cdot\ln x|\le 1/(e\epsilon)$ for
$x\in]0,1]$ and
 the Euler number $e=2,71...$. This leads to the estimate for  $z\in]0,1]$.
The estimate is trivial for $z\in]1,M]$.
\end{proof}

\begin{proof}[Proof of Theorem \ref{osc}:]
As the spherical functions and the moment functions on $G$ are constant on the
double cosets w.r.t.~$K$ by definition, and as each double coset has a  representative $g$
such that $gg^*=diag(a_1,\ldots,a_n)$ is diagonal with $a_1\ge\ldots\ge a_n>0$, it suffices to consider 
these group elements $g\in G$. We thus fix $\lambda\in\b R^n$ and $a_1\ge\ldots\ge a_n>0$ and put 
$a:=diag(a_1,\ldots,a_n)$. According to (\ref{int-rep}), (\ref{moment-function}) and (\ref{m1-vector})
we have to estimate
\begin{align}\label{diff1}
R:=& R(\lambda,a):=
|\phi_{i\rho+\lambda}(g)- e^{i\lambda\cdot m_{\bf 1}(g)}|
\\
=& \biggl|  \int_K exp\left(i\sum_{r=1}^n (\lambda_r-\lambda_{r+1})\cdot\ln\Delta_r(k^*ak)\right)\> dk
\notag
\\
&\quad\quad
 -exp\left(i \int_K\sum_{r=1}^n (\lambda_r-\lambda_{r+1})\cdot\ln\Delta_r(k^*ak)\> dk\right)
\biggr|
\notag
 \end{align}
with the convention $\lambda_{n+1}:=0$. For $r=1,\ldots,n$, we now use the polynomial
$C_r$ from Eq.~(\ref{C-r})
and write the logarithms of the principal minors in (\ref{diff1}) as 
\begin{equation}
\ln\Delta_r(k^*ak)=\ln C_r(a_1,\ldots,a_r) +\ln(H_r(k,a))
\quad\text{with}\quad
H_r(k,a):=\frac{\Delta_r(k^*ak)}{ C_r(a_1,\ldots,a_n)}.
 \end{equation}
With this notation and with $|e^{ix}|=1$ for $x\in\b R$, we rewrite (\ref{diff1}) as
\begin{align}\label{diff2}
R=&\biggl|  \int_K exp\left(i\sum_{r=1}^n (\lambda_r-\lambda_{r+1})\cdot\ln(H_r(k,a))\right)\> dk
\notag
\\
&\quad\quad
 -exp\left(i \int_K\sum_{r=1}^n (\lambda_r-\lambda_{r+1})\cdot\ln(H_r(k,a))\> dk\right)\biggr|.
\end{align}
We now use the power series for both exponential functions and observe that
 the terms of order 0 and 1 are equal in the difference above. Hence, 
$$R\le R_1+R_2$$
for
$$R_1:=\int_K\biggl| 
exp\left(i\sum_{r=1}^n (\lambda_r-\lambda_{r+1})\cdot\ln(H_r(k,a))\right)
-\left(1+i\sum_{r=1}^n (\lambda_r-\lambda_{r+1})\cdot\ln(H_r(k,a))\right)\biggr|\> dk$$
and
$$R_2:=\biggl|exp\left(i \int_K\sum_{r=1}^n (\lambda_r-\lambda_{r+1})\cdot\ln(H_r(k,a))\> dk\right)
\> -\> 1-i\int_K\sum_{r=1}^n (\lambda_r-\lambda_{r+1})\cdot\ln(H_r(k,a))\>
dk\biggr|.$$
Using the well-known elementary estimates $|\cos x-1|\le x^2/2$ and 
 $|\sin x-x|\le x^2/2$ for $x\in\b R$, we obtain
$|e^{ix}-(1+ix)|\le x^2$ for $x\in\b R$.
Therefore,
defining
$$A_m:=
\int_K\biggl| \sum_{r=1}^n
(\lambda_r-\lambda_{r+1})\cdot\ln(H_r(k,a))\biggr|^m\>dk
\quad\quad(m=1,2),$$
we conclude that
$$R\le R_1+R_2\le A_2 +A_1^2.$$
In the following, let $C_1,C_2,\ldots$ suitable constants. As $A_1^2\le A_2$  by 
 Jensen's inequality, and as
$$A_2\le \|\lambda\|^2 \cdot C_1\cdot
\int_K\sum_{r=1}^n |\ln(H_r(k,a))|^2\>dk =:\|\lambda\|^2 \cdot B_2,$$
we obtain
$R\le B_2\cdot 2 \|\lambda\|^2 $.
In order to complete the proof, we must check that $B_2$, i.e., that
 the integrals
\begin{equation}\label{I-m}
L_{r}:=\int_K |\ln(H_r(k,a))|^2\>dk
\end{equation}
remain bounded independent of 
$a_1,\ldots, a_n>0$ for $r=1,\ldots,n$.

For this fix  $r$.
 Lemma \ref{pos-coeff} in particular implies that for all $a_1,\ldots,a_n>0$,
$$\Delta_r(k^*ak)\le \sum_{1\le i_1<i_2<\ldots<i_r\le n}
a_{i_1}\cdot a_{i_2}\cdots \cdot a_{i_r} ={n\choose r}C_r(a_1,\ldots,a_n)$$
and $\Delta_r(k^*ak)>0$.
In other words, 
\begin{equation}\label{H-r}
0 <\frac{  \Delta_r(k^*ak) }{ C_r(a_1,\ldots,a_n)}=  H_r(k,a)  \le {n\choose r}.
\end{equation}
We conclude from (\ref{I-m}), (\ref{H-r}) and Lemma \ref{ln-est} that for any 
$\epsilon\in]0,1[$ and suitable $C_2=C_2(\epsilon)$,
$$L_{r}\le C_2\int_K \left(1+ H_r(a_1,\ldots,a_n)^{-\epsilon}\right)\> dk.$$
Therefore, by Lemma \ref{abs},
\begin{align}\label{suabs}
L_{r}\le& C_2 +C_3\int_K\left(  \sum_{1\le i_1<i_2<\ldots<i_r\le n} c_{i_1,\ldots,i_r}(k)^{-1}
  \right)^\epsilon\> dk
\notag\\
\le& C_2 +C_3\cdot {n\choose r}^\epsilon
 \sum_{1\le i_1<i_2<\ldots<i_r\le n}\int_K c_{i_1,\ldots,i_r}(k)^{-\epsilon}\> dk.
\end{align}
The right hand side of (\ref{suabs}) is independent of $a_1,\ldots,a_n$, and, by the definition of
 the $ c_{i_1,\ldots,i_r}(k)$ in Lemma \ref{pos-coeff}, $\int_K c_{i_1,\ldots,i_r}(k)^{-\epsilon}\> dk$
is independent of $1\le i_1<i_2<\ldots<i_r\le n$. Therefore, it suffices to check that
$$I_r:=\int_K c_{1,\ldots,r}(k)^{-\epsilon}\> dk =\int_K \Delta_r\left(k^*\left(\begin{array}{cc}
  I_r&0\\0&0\end{array}\right)k\right)^{-\epsilon}\> dk<\infty.$$
For this, we write $k$ as block matrix
$k=\left(\begin{array}{cc} k_r&*\\ *&k_{n-r}\end{array}\right)$ with
 $k_r\in M_r(\mathbb C)$ and $k_{n-r}\in M_{n-r}(\mathbb C)$ and observe that
$$ \Delta_r\left(k^*\left(\begin{array}{cc}
  I_r&0\\0&0\end{array}\right)k\right) =\Delta_r\left(\begin{array}{cc}k_r^*k_r
&*\\ *&*\end{array}\right)=|\det k_r|^2.
$$
We thus have to check that $\int_K |\det k_r|^{-2\epsilon}\> dk<\infty.$
As this is a consequence of the following lemma,
 the proof of the theorem is complete.
\end{proof}

 \begin{lemma}\label{int-finite} Keep the block matrix notation above. For $\epsilon<1/2$,
 $$\int_K |\det k_r|^{-2\epsilon}\> dk<\infty.$$
\end{lemma}

\begin{proof}
The statement is obvious for $r=n$.
 Moreover, by Lemma \ref{glech-det} we may assume that
$1\le r\le n/2$ which we shall assume now. In this case, we introduce the 
matrix ball 
$$B_r:= \{ w\in M_{r}(\b F):\> w^*w\le I_r\}$$ 
as well as the ball 
 $B:= \{ y \in M_{1,r}(\b F)\equiv \b F^n: \> \|y\|_2^2\le 1\}$. We
 conclude from the truncation lemma 2.1 of \cite{R2} that
$$\frac{1}{\kappa_r}\int_K |\det k_r|^{-2\epsilon}\> dk=\int_{B_r}|\det w|^{-2\epsilon}
\Delta(I_r-w^*w)^{(n-2r+1)\cdot d/2-1}\, dw$$
where $dw$ is the usual Lebesgue measure on the ball $B_r$ and
 $$\kappa_r:=\left( \int_{B_r} \det(I_r-w^*w)^{(n-2r+1)\cdot d/2-1}\> dw   \right)^{-1}.$$
Moreover,
by Lemma 3.7 and Corollary 3.8 of \cite{R1}, the mapping $P:B^r\to B_r$
with
\begin{equation}\label{trafo-P}
 P(y_1, \ldots, y_r):= \begin{pmatrix}y_1\\y_2(I_r-y_1^*y_1)^{1/2}\\ 
\vdots\\
  y_r(I_r-y_{r-1}^*y_{r-1})^{1/2}\cdots (I_r-y_{1}^*y_{1})^{1/2}\end{pmatrix}
\end{equation}
establishes a diffeomorphism such that the image of the measure 
$\det(I_r-w^*w)^{ (n-2r+1)\cdot d/2-1   }dw$
under  $P^{-1}$ is 
 $\,\prod_{j=1}^{r}(1-\|y_j\|_2^2)^{(n-r-j+1)\cdot d/2-1}dy_1\ldots
dy_r$. Moreover, we show in Lemma \ref{det-P} below that 
$$\det P(y_1, \ldots, y_r)=\det\left(\begin{array}{cc} y_1\\ \vdots\\y_r
\end{array}\right).$$
We thus conclude that
\begin{equation}
\int_K |\det k_r|^{-2\epsilon}\> dk=\frac{1}{\kappa_r} \int_B \ldots\int_B 
\left|\det\left(\begin{array}{cc} y_1\\ \vdots\\y_r
\end{array}\right)\right|^{-2\epsilon}
\prod_{j=1}^{r}(1-\|y_j\|_2^2)^{ (n-r-j+1)\cdot d/2-1   }dy_1\ldots
dy_r.
\end{equation}
This integral is finite for $\epsilon<1/2$, as 
one can use Fubini with an one-dimensional inner integral
 w.r.t.~the (1,1)-variable. After this inner integration, no further
 singularities appear  from the determinant-part in the remaining integral.
\end{proof}

\begin{lemma}\label{det-P} Keep the notations of the preceding proof. 
For all $y_1,\ldots,y_n\in B$,
$$\det P(y_1, \ldots, y_r)=\det\left(\begin{array}{cc} y_1\\ \vdots\\y_r
\end{array}\right).$$
\end{lemma}

\begin{proof}
Fix $y_1\in B$. The mapping  $y\mapsto y(I_r-y_1^*y_1)^{1/2}$ on $B$
has the following form: If $y$ is written as $y=ay_1+y^\perp$ in a unique way
with $a\in\b F$ and $y^\perp\perp y_1$, then
  $y(I_r-y_1^*y_1)^{1/2}=\sqrt{1-\|y_1\|_2^2}\cdot a y_1+y^\perp$ 
(write $I_r-y_1^*y_1$ in an orthonormal basis with $y_1/\|y_1\|_2$ as a member!). Using
  linearity of the determinant in all lines, we thus conclude that
$$\det\begin{pmatrix}y_1\\ y_2(I_r-y_1^*y_1)^{1/2}\\ 
\vdots\\
  y_r(I_r-y_{r-1}^*y_{r-1})^{1/2}\cdots (I_r-y_{1}^*y_{1})^{1/2}\end{pmatrix}
=\det\begin{pmatrix}y_1\\ y_2\\ y_3(I_r-y_2^*y_2)^{1/2}\\ 
\vdots\\ y_r(I_r-y_{r-1}^*y_{r-1})^{1/2}\cdots
 (I_r-y_{2}^*y_{2})^{1/2}\end{pmatrix}.$$
The lemma now follows by an obvious induction.
\end{proof}

\section{Moments and the proof of a  central limit theorem}

In this section we prove Theorem \ref{main-clt} and
 related results. We start with some facts about the moment functions of Section 1.
The first result concerns an estimate for $m_{\bf 1}$.

\begin{lemma}\label{absch-moment-function-1}
For $r=1,\ldots,n$ let
$$s_r(g):= m_{(1,0,\ldots,0)}(g)+\cdots+  m_{(0,\ldots,0, 1,0,\ldots,0)}(g) \quad(g\in G)$$
be the sum of the first $r$ moment functions of first order. Moreover, let
$\sigma_1(a)\ge \ldots\ge \sigma_n(a)>0$ be the ordered eigenvalues of a positive definite $n\times n$ matrix $a$. 
Then:
\begin{enumerate}
\item[\rm{(1)}] $s_n(g)=\ln\det(gg^*)$.
\item[\rm{(2)}] There is a constant $C=C(n)$ such that for all $r=1,\ldots,n$ and $g\in G$,
$$0\le \ln\sigma_1(gg^*)+\ldots+\ln\sigma_r(gg^*) \> -\> s_r(g)\le C.$$
\item[\rm{(3)}]  There is a constant $C=C(n)$ such that for all $g\in G$
 $$\left\| 2\ln\sigma_{sing}(g) - m_{\bf 1}(g)\right\|\le C.$$
\end{enumerate}
\end{lemma}

\begin{proof}
We may assume that $gg^*=diag(a_1,\ldots,a_n)$ with $a_l=\sigma_l(gg^*)$ ($l=1\ldots,n$).
 The integral representation (\ref{moment-function}) implies that 
$$s_r(g)=\int_K \ln\Delta_r(k^*gg^*k)\> dk.$$
This proves (1) and, in combination with Lemma \ref{pos-coeff}, the first inequality in (2). 
For the second  inequality of (2), we use the notations of Lemmas \ref{pos-coeff} and  \ref{abs}.
By the proof of Lemma \ref{abs}, we have for $k\in K$,
$$a_1\cdot a_2\ldots a_r\le {n \choose r} C_r(a_1,\ldots,a_n)\le
 \max_{1\le i_1<\ldots<i_r\le n} \frac{\ln\Delta_r(k^*gg^*k)}{c_{i_1,\ldots,i_r}(k)}.$$
Therefore,
$$\ln\sigma_1(gg^*)+\ldots+\ln\sigma_r(gg^*)=\int_K \ln(a_1\cdot a_2\ldots a_r)\> dk \le
\int_K \ln\Delta_r(k^*gg^*k)\> dk +M$$
for
 $$ M:=\int_K \max_{1\le i_1<\ldots<i_r\le n} \frac{1}{c_{i_1,\ldots,i_r}(k)}\> dk
\le \sum_{1\le i_1<\ldots<i_r\le n}\int_K\ln(c_{i_1,\ldots,i_r}(k)^{-1})\> dk.$$
As by the definition of $c_{i_1,\ldots,i_r}(k)$ all integrals in the sum are obviously equal,
 it suffices to show that
$$\int_K\ln(c_{1,\ldots,r}(k)^{-1})\> dk =-\int_K \ln\Delta_r\left(k^*
\left(\begin{array}{cc} I_r&0\\0&0\end{array}\right)k\right) \> dk$$
is finite. But this follows immediately from Lemma \ref{int-finite}. 
This proves (2). Finally, (3) is a consequence of (2).
\end{proof}

Lemma \ref{absch-moment-function-1}(3) implies that there exists $C=C(n)>0$
such that for all $g\in G$,
\begin{equation}
|e^{2i\lambda\cdot \ln\sigma_{sing}(g)}-e^{i\lambda\cdot m_{\bf 1}(g)}|\le C\cdot\|\lambda\|.
\end{equation}
Therefore, we conclude from Theorem \ref{osc}:

\begin{corollary}\label{cor-absch}
There exists a constant $C=C(n)>0$ such that for all $g\in G$,
$$\|\phi_{i\rho-\lambda}(g)-e^{2i\lambda\cdot \ln\sigma_{sing}(g)} \|\le  C\cdot\|\lambda\|.$$
\end{corollary}

\begin{remark}
It can be easily checked (e.g. for $n=2$ from explicit formulas in \cite{K})
that the uniform orders $\|\lambda\|^2$ and $\|\lambda\|$ in Theorem \ref{osc}
and Corollary \ref{cor-absch} respectively are sharp. We note that Corollary
\ref{cor-absch} is closely related to the Harish-Chandra expansion of the
spherical functions; see e.g. Opdam \cite{O} and Lemma I.4.2.2 of \cite{HS} in
the context of Heckman-Opdam hypergeometric functions which includes our
setting.
 We also remark  that in the proof of  the CLT \ref{main-clt} below Corollary
\ref{cor-absch} would be sufficient instead of the stronger Theorem
\ref{osc}. On the other hand, Theorem \ref{osc} leads generally to stronger
rates of convergence in the CLT; see  e.g. Theorem 4.2 of \cite{V2} for the
rank one case.
\end{remark}

 We shall also need the following estimate which follows immediately from the
 integral representation (\ref{int-rep}):

\begin{lemma}\label{absch-abl}
For all $g\in G$ and  $l\in\b N_0^n$,
$$\left|\frac{\partial^{|l|}}{\partial\lambda^l} 
\phi_{i\rho-\lambda}(g)\right|\le m_l(g).$$
\end{lemma}

Let $m\in\b N_0$ and $\nu\in M^1(G)$ a $K$-biinvariant probability measure. We
say that  $\nu$ admits finite  $m$-th modified moments if in the notation of
the introduction on the moment functions, 
$$m_{(m,0,\ldots,0)}, m_{(0,m,0,\ldots,0)},\ldots, m_{(0,\ldots,0,m)}\in L^1(G,\nu).$$
It follows immediately from (\ref{moment-function}) and H\"older's inequality that
in this case all moment functions of order at most $m$ are $\nu$-integrable.
Moreover, this moment condition implies a corresponding differentiability of
the spherical Fourier transform of  $\nu$:

\begin{lemma}\label{differentiable}
Let $m\in\b N_0$ and $\nu\in M^1(G)$ a $K$-biinvariant probability measure with finite $m$-th moments.
 Then the spherical Fourier transform
$$\tilde\nu:\b R^n\to\b C,\quad \lambda\mapsto\int_G \phi_{i\rho-\lambda}(g)\> d\nu(g)$$
is $m$-times continuously partially differentiable, and for all $l\in\b N_0^n$ with $|l|\le m$,
\begin{equation}\label{tauschallg}
\frac{\partial^{|l|}}{\partial\lambda^l}\tilde\nu(\lambda)=\int_G\frac{\partial^{|l|}}{\partial\lambda^l} 
\phi_{i\rho-\lambda}(g)\> d\nu(g).
\end{equation}
In particular,
\begin{equation}\label{tausch}
\frac{\partial^{|l|}}{\partial\lambda^l}\tilde\nu(0)=(-i)^{|l|} \int_G m_l(g)\> d\nu(g).
\end{equation}
\end{lemma}

\begin{proof} We proceed by induction: The case $m=0$ is trivial, and for $m\to m+1$
 we observe that by our assumption all moments of lower order exist, i.e., (\ref{tauschallg}) is available for
all $|l|\le m$. It now follows from Lemma \ref{absch-abl} and the well-known result about parameter integrals that
a further partial derivative and the integration can be interchanged. Finally, (\ref{tausch}) follows from
 (\ref{tauschallg}) and (\ref{moment-function}). Continuity of the derivatives
is also clear by Lemma \ref{absch-abl}.
\end{proof}

We next turn to the positive (semi)definiteness of the modified covariance
 matrix $\sigma^2(\nu)$
for biinvariant measures with finite second modified moments. We start with measures concentrated on a 
double coset:

\begin{lemma}\label{posdef-moment-function} Let $n\ge2$, $g\in G$, and
 $\Sigma^2(g):=m_{\bf 2}(g)-m_{\bf 1}(g)^t m_{\bf 1}(g)$.
\begin{enumerate}
\item[\rm{(1)}]
$\Sigma^2(g)$
 is positive semidefinite.
\item[\rm{(2)}] If  $gg^*$ is not a multiple of the identity matrix, then 
$\Sigma^2(g)$ has rank $n-1$. 
\item[\rm{(3)}] If   $gg^*$ is  a multiple of the identity matrix, then 
$\Sigma^2(g)=0$.
\end{enumerate}
\end{lemma}

\begin{proof}
Let $a_1,\ldots,a_n\in\b R$ with $a_1^2+\ldots+ a_n^2>0$ and the row vector $a=(a_1,\ldots,a_n)$.
Put
$$f_1(k):=\ln\Delta_1(k^*gg^*k)\quad \text{and}
\quad
f_l(k):=\ln\Delta_l(k^*gg^*k)-\ln\Delta_{l-1}(k^*gg^*k)\quad(l=2,\cdots,n).$$
Then, by (\ref{moment-function}), (\ref{m1-vector}),  (\ref{m2-matrix}), and the Cauchy-Schwarz inequality,
$$a\left( m_{\bf 2}(g)-m_{\bf 1}(g)^t m_{\bf 1}(g)\right)a^t
= \int_K \left(\sum_{l=1}^n a_l f_l(k)\right)^2\> dk -\left(\int_K \sum_{l=1}^n a_l f_l(k)\> dk\right)^2
\ge0.$$
Moreover, this expression is equal to $0$ if and only if the function 
$$k\mapsto \sum_{l=1}^n a_l f_l(k)= (a_1-a_2)\ln\Delta_1(k^*gg^*k)+\dots+ (a_{n-1}-a_n)\ln\Delta_{n-1}(k^*gg^*k)
+a_n\ln\Delta_n(k^*gg^*k)$$
 is constant on $K$. As $k\mapsto\ln\Delta_n(k^*gg^*k)$ is constant on $K$, and as under the condition of (2),
the functions $k\mapsto\ln\Delta_r(k^*gg^*k)$ ($r=1,\ldots,n-1$) and the constant function $1$ are linearly
 independent on $K$ by Corollary \ref{lin-unab} in the appendix, the function 
 $k\mapsto \sum_{l=1}^n a_l f_l(k)$ is constant on $K$ precisely for $a_1=a_2=\ldots=a_n$. This proves (2).
Part (3) is obvious.
\end{proof}

The arguments of the preceding proof lead to the following characterization of
$K$-biinvariant measures with positive definite covariance matrices:

\begin{lemma}\label{pos-def-co} Let $\nu\in M^1(G)$ be a $K$-biinvariant probability measure
  having second modified moments. Then $\Sigma^2(\nu)$ is positive definite if
 and only if  $supp\> \nu$ is not contained in the subgroups $\{cI_n: \> c\in\b
 F, \> c\ne0\}$
  and  $SL_n(\b F)$.
\end{lemma}

We now turn to the proof of the CLT:

 \begin{proof}[Proof of  Theorem \ref{main-clt}]
 Let $\nu\in M^1(G)$ be a $K$-biinvariant probability measure with finite second modified moments.
Let $(X_k)_{k\ge1}$ be i.i.d.~$G$-valued 
random variables with distribution $\nu$ and $S_k:=X_1\cdot X_2\cdots X_k$. Let $\lambda\in\b R^n$.
As the functions $\phi_{i\rho-\lambda}$ are bounded on $G$ (by the integral representation (\ref{int-rep}))
and multiplicative w.r.t.~$K$-biinvariant measures, we have
$$E(\phi_{i\rho-\lambda/\sqrt k}(S_k))= \int_G \phi_{i\rho-\lambda/\sqrt k}(g)\> d\nu^{(k)}(g)=
\left(\int_G \phi_{i\rho-\lambda/\sqrt k}(g)\> d\nu(g)\right)^k =\tilde\nu(\lambda/\sqrt k)^k.$$
We now use Taylor's formula, Lemma \ref{differentiable}, and
$$m_{\bf 2}(\nu):=\int_G m_{\bf 2}(g)\> d\nu(g)=\Sigma^2(\nu)+m_{\bf 1}(\nu)^tm_{\bf 1}(\nu)$$
and obtain
\begin{align}
exp(&i\lambda\cdot m_{\bf 1}(\nu) \sqrt k)\cdot E(\phi_{i\rho-\lambda/\sqrt k}(S_k))\quad =\quad
\left(exp(i\lambda\cdot m_{\bf 1}(\nu)/ \sqrt k)\cdot\tilde\nu(\lambda/\sqrt k)\right)^k
\\
=&\left(\left[1+\frac{i \lambda\cdot m_{\bf 1}(\nu)}{ \sqrt k}- 
\frac{(\lambda\cdot m_{\bf 1}(\nu))^2}{ 2 k} +o(\frac{1}{k})\right]\cdot
\left[1-\frac{i \lambda\cdot m_{\bf 1}(\nu)}{ \sqrt k}-\frac{\lambda m_{\bf 2}(\nu)\lambda^t}{ 2 k} +o(\frac{1}{k})\right]
\right)^k
\notag \\
=&\biggl(\left[1+\frac{i \lambda\cdot m_{\bf 1}(\nu)}{ \sqrt k}- 
\frac{(\lambda\cdot m_{\bf 1}(\nu))^2}{ 2 k} +o(\frac{1}{k})\right]\cdot
\notag \\
&\quad\quad \times\left[1-\frac{i \lambda\cdot m_{\bf 1}(\nu)}{ \sqrt k}-
\frac{\lambda (\Sigma^2(\nu)+ m_{\bf 1}(\nu)^t m_{\bf 1}(\nu))\lambda^t}{ 2 k} +o(\frac{1}{k})\right]
\biggr)^k
\notag \\
=&\left( 1-\frac{\lambda \Sigma^2(\nu)\lambda^t}{ 2 k} +o(\frac{1}{k})\right)^k.
\notag
\end{align}
Therefore,
$$\lim_{k\to\infty} exp(i\lambda\cdot m_{\bf 1}(\nu) \sqrt k)\cdot E(\phi_{i\rho-\lambda/\sqrt k}(S_k))
= exp(-\lambda  \Sigma^2(\nu)\lambda^t/ 2 ).$$
Moreover, by Theorem \ref{osc},
$$\lim_{k\to\infty} E\left( \phi_{i\rho-\lambda/\sqrt k}(S_k)- exp(-i\lambda\cdot m_{\bf 1}(S_k)/\sqrt k)\right) =0.$$
We conclude that
$$\lim_{k\to\infty} exp(-i\lambda\cdot ( m_{\bf 1}(S_k)- k\cdot m_{\bf 1}(\nu) )/\sqrt k)=
 exp(-\lambda  \Sigma^2(\nu)\lambda^t/ 2 )$$
for all $\lambda\in\b R^n$. Levy's continuity theorem for the classical $n$-dimensional
 Fourier transform now implies that
$( m_{\bf 1}(S_k)-k\cdot m_{\bf 1}(\nu))/\sqrt k$ tends in distribution to $N(0, \Sigma^2(\nu))$.
By the estimate of Lemma \ref{absch-moment-function-1}(1), this immediately implies Theorem
\ref{main-clt}.
\end{proof}

On the basis of Theorem \ref{osc}, also a  Berry-Esseen-type estimate with
 the order $ O(k^{-1/3})$ of convergence  can be derived.
 As the details are technical, but quite similar to the proof 
of the corresponding rank-one-case in Theorem 4.2 of \cite{V2}, we here omit
details.
We also mention that Theorem \ref{osc} can be also used to derive further CLTs
e.g. with stable distributions with domains of attraction or a Lindeberg-Feller CLT. The details of proof
then would be also very similar to the classical cases for sums of iid random
variables.

\section{Appendix}

Here we collect some results from linear algebra which are needed in Section 3.

\begin{lemma}\label{det-rechnung}
Let $x_1,\ldots,,x_n\in\b R$. Then 
\begin{align}
\det&\left(\begin{array}{cccccc}  x_1&x_2&x_3&x_4&\cdots&x_n
\\ x_1+x_2& x_2+x_1&  x_3+x_2&  x_4+x_2&\cdots&x_n+x_2
\\x_1+x_2+x_3& x_2+x_1+x_3&  x_3+x_2+x_1&  x_4+x_2+x_3&\cdots&x_n+x_2+x_3
\\ \vdots &\vdots &\vdots &\vdots &\vdots &\vdots \\
\sum_{l=1}^nx_l &\sum_{l=1}^nx_l&\sum_{l=1}^nx_l&\sum_{l=1}^nx_l&\cdots&
\sum_{l=1}^nx_l
 \end{array}\right) =
\notag\\
&=(x_1+x_2+\cdots+ x_n)\cdot(x_1-x_2)\cdot(x_1-x_3)\cdots(x_1-x_n)
\notag
\end{align}
\end{lemma}

\begin{proof} The determinant is a homogeneous polynomial in the the variables $x_1,\ldots,x_n$  of degree $n$.  
Moreover, the monomial $x_1^n$ appears in this polynomial with coefficient 1, and 
for given $x_2,\ldots,x_n$, the determinant is a polynomial in the variable $x_1$ where 
 $-(x_2+\cdots+ x_n)$, $x_2$, $x_3$,\ldots, $x_n$ are the zeros of this polynomial. This leads readily to the claim.
\end{proof} 

\begin{corollary}\label{lin-unab}
Let $a_1,\ldots,a_n>0$ numbers such that at least two of them are different.
 Consider the diagonal matrix $a=diag(a_1,\ldots,a_n)$.
 Then the functions $k\mapsto\ln\Delta_r(k^*ak)$  
with $r=1,\ldots,n-1$ and the constant function 1 on $K=U_n(\b F)$  are linearly independent.
\end{corollary}

\begin{proof}
Without loss of generality, $a_1$ is different from $a_2,\ldots,a_n$. Now consider the $n$
permutation matrices $k_l$ which permute the rows $1$ and $l$ and leave the other rows invariant for $l=1,\cdots,n$.
Then, using the notation $x_l:=\ln a_l$, the number
$\ln\Delta_r(k_j^*ak_j)$ is precisely the $r,l$-entry of the matrix in Lemma \ref{det-rechnung}.
Therefore, by Lemma \ref{det-rechnung}, $\det((\ln\Delta_r(k_j^*ak_j))_{r,j=1,\ldots,n})\ne0$
 for $x_1+\ldots +x_n\ne0$, i.e., $a$ with $\det a\neq 1$. As $\ln\Delta_n(k^*ak)$ is constant,
this proves the statement of the corollary for 
$\det a\neq1$. The case $\det a=1$ can be easily derived by 
considering $2a$ instead of $a$ in the preceding argument.
\end{proof}

\end{document}